\documentclass[reqno]{amsart}
\usepackage{graphicx}

\usepackage{epstopdf}
\usepackage[justification=centering]{caption}
\usepackage{lineno,hyperref}
\usepackage{amsfonts,amsmath,latexsym,amssymb,amsthm}
\usepackage{geometry}
\newcommand{\rs}{{\mathcal R\,}}
\newcommand{\hs}{{\mathcal H\,}}
\newcommand{\ws}{{\mathcal W\,}}

\modulolinenumbers[5]
\usepackage{makecell,multirow,diagbox}
\hypersetup{
	colorlinks=true,
	linkcolor=blue
}
\hypersetup{
	citecolor=blue
}
\newtheorem{definition}{Definition}
\newtheorem{theorem}{Theorem}
\newtheorem{lemma}{Lemma}
\usepackage[justification=centering]{caption}
\newtheorem{corollary}{Corollary}

\theoremstyle{remark}
\newtheorem{remark}[theorem]{Remark}

\numberwithin{equation}{section}
\allowdisplaybreaks[4]

\begin{document}

\title{New inequalities for weaving frames in Hilbert spaces}

\author{Dongwei Li}
\address{School of Mathematics, HeFei University of Technology, 230009, P. R. China}
\email{dongweili@huft.edu.cn}

\author{Jinsong Leng}
\address{School of Mathematical Sciences, University of Electronic Science and Technology of China, 611731, P. R. China}
\email{jinsongleng@126.com}




\begin{abstract}
	In this paper, we establish Parseval identities and surprising new inequalities for weaving frames in Hilbert space, which involve scalar $\lambda\in\rs$. By suitable choices of  $\lambda$, one obtains the previous results as special cases. Our results generalize and improve the remarkable results which have been obtained by Balan et al. and G\u{a}vru\c{t}a.

\end{abstract}

\maketitle

\section{Introduction}
Frames in Hilbert spaces were first introduced in 1952 by Duffin and Schaeffer \cite{duffin1952class} to study some deep problems in nonharmonic Fourier series,  reintroduced in 1986 by Daubechies, Grossmann and Meyer \cite{daubechies1986painless}, and today frames play important roles in many applications in several areas of mathematics, physics, and engineering, such as coding theory \cite{leng2011optimal,li2018frame}, sampling theory \cite{zhao2006perturbation,sun2017nonlinear}, quantum measurements \cite{eldar2002optimal},filter bank theory \cite{kovacevic2002filter} and image processing \cite{donoho2005continuous}.

Let $\hs$ be a separable space and $I$ a countable index set. A sequence  $\{\phi_i\}_{i\in I}$ of elements of $\hs$ is a frame for $\hs$ if there exist constants $A,~B>0$ such that
$$A\|f\|^2\le \sum_{i\in I}\left| \left\langle f,\phi_i\right\rangle \right| ^2\le B\|f\|^2,~~~\forall f\in\hs.$$
The number $A,~B$ are called lower and upper frame bounds, respectively. If $A=B$, then this frame is called an $A$-tight frame, and if $A=B=1$, then it is called a Parseval frame.

Suppose  $\{\phi_i\}_{i\in I}$ is a frame for $\hs$, then the frame operator is a self-adjoint positive invertible operators, which is given by
$$S:\hs\rightarrow\hs,~~Sf=\sum_{i\in I}\left\langle f,\phi_i\right\rangle \phi_i.$$ 
The following reconstruction formula holds:
$$f=\sum_{i\in I}\left\langle f,\phi_i\right\rangle S^{-1}\phi_i=\sum_{i\in I}\left\langle f,S^{-1}\phi_i\right\rangle \phi_i,$$
where the family $\{\widetilde{\phi}_i\}_{i\in I}=\{S^{-1}\phi_i\}_{i\in I}$ is also a frame for $\hs$, which is called the canonical dual frame of $\{\phi_i\}_{i\in I}$.
The frame $\{\varphi_i\}_{i\in I}$ for $\hs$ is called an alternate dual frame of $\{\phi_i\}_{i\in I}$ if the following formula holds:
$$f=\sum_{i\in I}\left\langle f,\phi_i\right\rangle \varphi_i=\sum_{i\in I}\left\langle f,\varphi_i\right\rangle \phi_i$$
for all $f\in\hs$ \cite{han2000frames}.

Let $\{\phi_i\}_{i\in I}$ be a frame for $\hs$, for every $J\subset I$, we define the operator
$$S_J=\sum_{i\in J}\left\langle f,\phi_i\right\rangle \phi_i,$$
and denote $J^c=I\setminus J$.

The concept of discrete weaving frames for separable Hilbert spaces was introduced by Bemrose, Casazza et. al. \cite{bemrose2015weaving}, which is motivated by distributed signal processing. For example, in wireless sensor networks where frames may be subjected to distributed processing under different frames. Thus, weaving frames have potential applications in wireless sensor networks that require distributed processing under different frames, as well as preprocessing of signals using Gabor frames. Recently, weaving frames in Hilbert spaces have been studied intensively, for more details see \cite{casazza2016weaving,khosravi2018weaving,vashisht2016weaving}

\begin{definition}
Two frames $\{\phi_i\}_{i\in I}$ and $\{\psi_i\}_{i\in I}$  for a Hilbert space $\hs$ is said to be woven if there are universal constants $A$ and $B$ so that for every partition $\sigma\subset I$, the family $\{\phi_i\}_{i\in\sigma}\cup\{\psi_i\}_{i\in\sigma^c}$ is a frame for $\hs$ with lower and upper frame bounds $A$ and $B$, respectively. The family $\{\phi_i\}_{i\in\sigma}\cup\{\psi_i\}_{i\in\sigma^c}$ is called a weaving. 
\end{definition}
If $A=B$, we call that $\{\phi_i\}_{i\in I}$ and $\{\psi_i\}_{i\in I}$ are $A$-woven, and if $A=B=1$, then we call them $1$-woven.  

Suppose that $\{\phi_i\}_{i\in I}$ and $\{\psi_i\}_{i\in I}$  are woven, the frame operator of $\{\phi_i\}_{i\in\sigma}\cup\{\psi_i\}_{i\in\sigma^c}$ is defined by
$$S_{\ws}f=\sum_{i\in\sigma}\left\langle f,\phi_i\right\rangle \phi_i+\sum_{i\in\sigma^c}\left\langle f,\psi_i\right\rangle \psi_i,$$
then $S_{\ws}$ is a bounded, invertible, self-adjoint and positive operator. A frame  $\{\varphi_i\}_{i\in I}$ is called an alternate dual frame of $\{\phi_i\}_{i\in\sigma}\cup\{\psi_i\}_{i\in\sigma^c}$ if for all $f\in\hs$ the following identity holds:
\begin{equation}\label{1.1}
f=\sum_{i\in \sigma}\left\langle f,\phi_i\right\rangle \varphi_i+\sum_{i\in \sigma^c}\left\langle f,\psi_i\right\rangle \varphi_i.
\end{equation}
For every $\sigma\subset I$, define the bounded linear operators $S_{\ws}^{\sigma},~S_{\ws}^{\sigma^c}:\hs\rightarrow\hs$ by 
$$S_{\ws}^{\sigma}f=\sum_{i\in\sigma}\left\langle f,\phi_i\right\rangle \phi_i,~~S_{\ws}^{\sigma^c}f=\sum_{i\in\sigma^c}\left\langle f,\psi_i\right\rangle \psi_i.$$
It is easy to check that 
 $S_{\ws}^{\sigma}$ and $S_{\ws}^{\sigma^c}$ are  self-adjoint.
 
 In \cite{balan2006signal}, the authors solved a long-standing conjecture of the signal processing community.  They showed that for suitable frames $\{\phi_i\}_{i\in I}$, a signal $f$ can (up to a global phase) be recovered from the phase-less measurements $\{|\langle f,\phi_i\rangle|\}_{i\in I}$. Note, that this only shows that reconstruction of $f$ is in principle possible, but there is not an effective constructive algorithm. While searching for such an algorithm, the authors of \cite{balan2005decompositions} discovered a new identity for Parseval frames \cite{balan2007new}. The authors in \cite{guavructa2006some,zhu2010note} generalized these identities to alternate dual frames and got some general results.
 The study of inequalities has interested many mathematicians.  Some authors have extended the equalities and inequalities for frames in Hilbert spaces to generalized  frames \cite{li2018some,li2012some,xiao2015erasures,poria2017some}. The following form was given in \cite{balan2007new} (See \cite{balan2005decompositions} for a discussion of the origins of this fundamental identity).
 
\begin{theorem}\label{thm1}
Let $\{\phi_i\}_{i\in I}$ be a Parseval frame for $\hs$. For every $J\subset I$ and every $f\in\hs$, we have
\begin{equation}\label{1.2}
\sum_{i\in J}\left| \left\langle f,\phi_i\right\rangle \right| ^2+\left\| \sum_{i\in J^c}\left\langle f,\phi_i\right\rangle \phi_i\right\| ^2=\sum_{i\in J^c}\left| \left\langle f,\phi_i\right\rangle \right| ^2+\left\| \sum_{i\in J}\left\langle f,\phi_i\right\rangle \phi_i\right\| ^2\ge \frac{3}{4}\|f\|^2.
\end{equation}
\end{theorem}
 
 Later on, the author in \cite{guavructa2006some} generalized Theorem \ref{thm1} to general frames.
 
\begin{theorem}\label{thm2}
	Let $\{\phi_i\}_{i\in I}$ be a frame for $\hs$ with canonical dual frame $\{\widetilde{\phi}_i\}_{i\in I}$. Then for every $J\subset I$ and every $f\in\hs$, we have
	\begin{equation}\label{1.3}
	\sum_{i\in J}\left| \left\langle f,\phi_i\right\rangle \right| ^2+\sum_{i\in I}\left| \left\langle S_{J^c}f,\widetilde{\phi}_i\right\rangle \right| ^2=\sum_{i\in J^c}\left| \left\langle f,\phi_i\right\rangle \right| ^2+\sum_{i\in I}\left| \left\langle S_{J}f,\widetilde{\phi}_i\right\rangle \right| ^2\ge \frac{3}{4}\sum_{i\in I}\left| \left\langle f,\phi_i\right\rangle \right| ^2.
	\end{equation}
\end{theorem}
\begin{theorem}\label{thm2.1}
	Let $\{\phi_i\}_{i\in I}$ be a frame for $\hs$ and $\{\varphi_i\}_{i\in I}$ be an alternate dual frame of $\{\phi_i\}_{i\in I}$. Then for every $J\subset I$ and every $f\in\hs$, we have
	\begin{equation}\label{1.4}
{\rm Re}\left( \sum_{i\in J}\left\langle f,\varphi_i\right\rangle \overline{\left\langle f,\phi_i\right\rangle }\right) +\left\| \sum_{i\in J^c}\left\langle f,\varphi_i\right\rangle \phi_i\right\| ^2={\rm Re}\left( \sum_{i\in J^c}\left\langle f,\varphi_i\right\rangle \overline{\left\langle f,\phi_i\right\rangle }\right) +\left\| \sum_{i\in J}\left\langle f,\varphi_i\right\rangle \phi_i\right\| ^2\ge \frac{3}{4}\|f\|^2.
	\end{equation}
\end{theorem}

Motivated by these interesting results, the authors in \cite{zhu2010note} generalized the Theorem \ref{thm2.1} to a more general form which does not involve the real parts of the complex numbers.

\begin{theorem}\label{thm2.2}
	Let $\{\phi_i\}_{i\in I}$ be a  frame for $\hs$ and $\{\varphi_i\}_{i\in I}$ be an alternate dual frame of $\{\phi_i\}_{i\in I}$. Then for every $J\subset I$ and every $f\in\hs$, we have
	\begin{equation}\label{1.5}
	\left( \sum_{i\in J}\left\langle f,\varphi_i\right\rangle \overline{\left\langle f,\phi_i\right\rangle }\right) +\left\| \sum_{i\in J^c}\left\langle f,\varphi_i\right\rangle \phi_i\right\| ^2=\overline{\left( \sum_{i\in J^c}\left\langle f,\varphi_i\right\rangle \overline{\left\langle f,\phi_i\right\rangle }\right)} +\left\| \sum_{i\in J}\left\langle f,\varphi_i\right\rangle \phi_i\right\| ^2\ge \frac{3}{4}\|f\|^2.
	\end{equation}
\end{theorem}

In this paper, we generalize the above mentioned results for weaving frames in Hilbert spaces. We generalize the above inequalities to a more general form which involve a scalar $\lambda\in\rs$ which is different from the scalar $\lambda\in[0,1]$ in \cite{poria2017some}. Since a frame is woven with itself, the previous equality and inequalities in frames can be obtained as a special case of the results we establish on weaving frames.
\section{Results and their proofs}
We first state a simple result on operators, which is  a distortion of \cite[Lemma 2.1]{zhu2010note}.
\begin{lemma}\label{lem1}
If $P,~Q\in L(\hs)$ satisfying $P+Q=I_{\hs}$, then $P+Q^*Q=Q^*+P^*P$.
\end{lemma}
\begin{proof}
	A simple computation shows that
	$$P+Q^*Q=I_{\hs}-Q+Q^*Q=I_{\hs}-(I_{\hs}-Q^*)Q=I_{\hs}-P^*(I_\hs-P)=I_{\hs}-P^*+P^*P=Q^*+P^*P.$$
\end{proof}
Now we state and prove a Parseval weaving frame identity.
\begin{theorem}\label{thm5.1}
Suppose $\{\phi_i\}_{i\in I}$ and $\{\psi_i\}_{i\in I}$  for a Hilbert space $\hs$ are 1-woven. Then for all  $\sigma\subset I$  and all $f\in\hs$, we have
\begin{equation}\label{2.1}
\sum_{i\in\sigma}\left| \left\langle f,\phi_i\right\rangle \right| ^2+\left\| \sum_{i\in\sigma^c}\left\langle f,\psi_i\right\rangle \psi_i\right\| ^2=\sum_{i\in\sigma^c}\left| \left\langle f,\psi_i\right\rangle \right| ^2+\left\| \sum_{i\in\sigma}\left\langle f,\phi_i\right\rangle \phi_i\right\| ^2\ge\frac{3}{4}\|f\|^2.
\end{equation}
\end{theorem}
\begin{proof}
Since $\{\phi_i\}_{i\in I}$ and $\{\psi_i\}_{i\in I}$ are  $1$-woven,  the weaving frame $\{\phi_i\}_{i\in\sigma}\cup\{\psi_i\}_{i\in\sigma^c}$ is a Parseval frame for $\hs$. Then the frame operator of $\{\phi_i\}_{i\in\sigma}\cup\{\psi_i\}_{i\in\sigma^c}$  is $S_{\ws}=I_{\hs}$. For every $\sigma\subset I$, we have $S_{\ws}^{\sigma}+S_{\ws}^{\sigma^c}=I_{\hs}$. Note that $S_{\ws}^{\sigma^c}$ is a self-adjoint operator, and therefore $(S_{\ws}^{\sigma^c})^*=S_{\ws}^{\sigma^c}$. By applying Lemma \ref{lem1} to the operator $S_{\ws}^{\sigma}$ and $S_{\ws}^{\sigma^c}$, for all $f\in\hs$, we obtain
$$\left\langle S_{\ws}^{\sigma}f,f\right\rangle +\left\langle (S_{\ws}^{\sigma^c})^*S_{\ws}^{\sigma^c}f,f\right\rangle=
\left\langle (S_{\ws}^{\sigma^c})^*f,f\right\rangle +\left\langle (S_{\ws}^{\sigma})^*S_{\ws}^{\sigma}f,f\right\rangle.  $$
Thus
$$\left\langle S_{\ws}^{\sigma}f,f\right\rangle +\left\| S_{\ws}^{\sigma^c}f\right\|=
\left\langle S_{\ws}^{\sigma^c}f,f\right\rangle +\left\|  S_{\ws}^{\sigma}f\right\| . $$
Hence
$$\sum_{i\in\sigma}\left| \left\langle f,\phi_i\right\rangle \right| ^2+\left\| \sum_{i\in\sigma^c}\left\langle f,\psi_i\right\rangle \psi_i\right\| ^2=\sum_{i\in\sigma^c}\left| \left\langle f,\psi_i\right\rangle \right| ^2+\left\| \sum_{i\in\sigma}\left\langle f,\phi_i\right\rangle \phi_i\right\| ^2.$$
Next, we prove the inequality of \eqref{2.1}. A simple computation shows that 
$$(S_{\ws}^{\sigma})^2+(S_{\ws}^{\sigma^c})^2=(S_{\ws}^{\sigma})^2+(I_{\hs}-(S_{\ws}^{\sigma}))^2=2(S_{\ws}^{\sigma})^2-2S_{\ws}^{\sigma}+I_{\hs}=2\left( S_{\ws}^{\sigma}-\frac{1}{2}I_{\hs}\right) ^2+\frac{1}{2}I_{\hs},$$
and so 
$$(S_{\ws}^{\sigma})^2+(S_{\ws}^{\sigma^c})^2\ge \frac{1}{2}I_{\hs}.$$
Since $S_{\ws}^{\sigma}+S_{\ws}^{\sigma^c}=I_{\hs}$. it follows that 
\begin{equation}\label{2.2}
S_{\ws}^{\sigma}+(S_{\ws}^{\sigma^c})^2+S_{\ws}^{\sigma^c}+(S_{\ws}^{\sigma})^2\ge \frac{3}{2}I_{\hs}.
\end{equation}
Notice that operator $S_{\ws}^{\sigma}$ is also self-adjoint and therefore $(S_{\ws}^{\sigma})^*=S_{\ws}^{\sigma}$. Applying Lemma \ref{lem1} to the operators $P=S_{\ws}^{\sigma}$ and $Q=S_{\ws}^{\sigma^c}$, we obtain
$$S_{\ws}^{\sigma}+(S_{\ws}^{\sigma^c})^2=S_{\ws}^{\sigma^c}+(S_{\ws}^{\sigma})^2.$$
Then equation \eqref{2.2} means that
$$(S_{\ws}^{\sigma^c}+(S_{\ws}^{\sigma})^2)\ge\frac{3}{4}I_{\hs}.$$
Therefore for all $f\in\hs$, we have
$$\sum_{i\in\sigma^c}\left| \left\langle f,\psi_i\right\rangle \right| ^2+\left\| \sum_{i\in\sigma}\left\langle f,\phi_i\right\rangle \phi_i\right\| ^2=\left\langle S_{\ws}^{\sigma^c}f,f\right\rangle +\left\langle S_{\ws}^{\sigma}f,S_{\ws}^{\sigma}f\right\rangle=\left\langle (S_{\ws}^{\sigma^c}+(S_{\ws}^{\sigma})^2)f,f\right\rangle \ge\frac{3}{4}\|f\|^2.$$
This completes the proof. 
\end{proof}
\begin{remark}
If we take $\phi_i=\psi_i$ for all $i\in I$ in Theorem \ref{thm5.1}, we can obtain the Theorem \ref{thm1}.
\end{remark}

\begin{lemma}\label{lem2}
Let $P,~Q\in L(\hs)$ be two self-adjoint operators such that $P+Q=I_{\hs}$. Then for any $\lambda\in\rs$, and all $f\in\hs$, we have
	\begin{align*}
	\|Pf\|^2+\lambda\left\langle Qf,f\right\rangle &=\|Qf\|^2+(2-\lambda)\left\langle Pf,f\right\rangle +(\lambda-1)\|f\|^2\\
&\ge (\lambda-\frac{\lambda^2}{4})\|f\|^2.
\end{align*}
\end{lemma}
\begin{proof}
	For all $f\in\hs$, we have
		\begin{align}\label{2.3}
	\|Pf\|^2+\lambda\left\langle Qf,f\right\rangle =&\left\langle P^2f,f\right\rangle +\lambda\left\langle (I_{\hs}-P)f,f\right\rangle\nonumber\\ 
	=&\left\langle (P^2-\lambda P+\lambda I_{\hs})f,f\right\rangle  \\
	=&\left\langle (I_{\hs}-P)^2f,f\right\rangle+(2-\lambda)\left\langle Pf,f\right\rangle +(\lambda-1)\left\langle f,f\right\rangle  \nonumber\\
	=&\|Qf\|^2+(2-\lambda)\left\langle Pf,f\right\rangle +(\lambda-1)\|f\|^2.\nonumber
	\end{align}
	 A simple computation of   \eqref{2.3}, we have
		\begin{align*}
\left\langle (P^2-\lambda P+\lambda I_{\hs})f,f\right\rangle
=&\left\langle ((P-\frac{\lambda}{2} I_{\hs})^2-\frac{\lambda^2}{4}I_{\hs}+\lambda I_{\hs})f,f\right\rangle  \\
=&\left\langle ((P-\lambda I_{\hs})^2+(\lambda-\frac{\lambda^2}{4})I_{\hs})f,f\right\rangle\\
\ge& (\lambda-\frac{\lambda^2}{4})\|f\|^2.
\end{align*}	
 This proves the desired result.
\end{proof}
\begin{theorem}\label{thm7.1}
Suppose two frames $\{\phi_i\}_{i\in I}$ and $\{\psi_i\}_{i\in I}$  for a Hilbert space $\hs$ are woven. Then for any $\lambda\in\rs$, for all $\sigma\subset I$ and all $f\in\hs$, we have
	\begin{align*}
&\sum_{i\in\sigma}\left| \left\langle f,\phi_i\right\rangle \right| ^2+\sum_{i\in\sigma}\left| \left\langle S^{\sigma^c}_{\ws}f,S_{\ws}^{-1}\phi_i\right\rangle \right| ^2+\sum_{i\in\sigma^c}\left| \left\langle S^{\sigma^c}_{\ws}f,S_{\ws}^{-1}\psi_i\right\rangle \right| ^2\\
=&\sum_{i\in\sigma^c}\left| \left\langle f,\psi_i\right\rangle \right| ^2+\sum_{i\in\sigma}\left| \left\langle S^{\sigma}_{\ws}f,S_{\ws}^{-1}\phi_i\right\rangle \right| ^2+\sum_{i\in\sigma^c}\left| \left\langle S^{\sigma}_{\ws}f,S_{\ws}^{-1}\psi_i\right\rangle \right| ^2\\
\ge &(\lambda-\frac{\lambda^2}{4})\sum_{i\in\sigma}\left| \left\langle f,\phi_i\right\rangle \right| ^2+(1-\frac{\lambda^2}{4})\sum_{i\in\sigma^c}\left| \left\langle f,\psi_i\right\rangle \right| ^2.
\end{align*}
\end{theorem}
\begin{proof}
Since $\{\phi_i\}_{i\in I}$ and $\{\psi_i\}_{i\in I}$ are woven, for all $\sigma\subset I$, $\{\phi_i\}_{i\in\sigma}\cup\{\psi_i\}_{i\in\sigma^c}$ is a frame for $\hs$. Let $S_{\ws}$ be the frame operator for $\{\phi_i\}_{i\in\sigma}\cup\{\psi_i\}_{i\in\sigma^c}$. Since $S_{\ws}^{\sigma}+S_{\ws}^{\sigma^c}=S_{\ws}$, it follows that
$$S_{\ws}^{-1/2}S_{\ws}^{\sigma}S_{\ws}^{-1/2}+S_{\ws}^{-1/2}S_{\ws}^{\sigma^c}S_{\ws}^{-1/2}=I_{\hs}.$$
Considering $P=S_{\ws}^{-1/2}S_{\ws}^{\sigma}S_{\ws}^{-1/2}$, $Q=S_{\ws}^{-1/2}S_{\ws}^{\sigma^c}S_{\ws}^{-1/2}$, and $S_{\ws}^{1/2}f$ instead of $f$ in Lemma \ref{lem2}, we obtain
		\begin{align*}
&\|S_{\ws}^{-1/2}S_{\ws}^{\sigma}f\|^2+\lambda\left\langle S_{\ws}^{-1/2}S_{\ws}^{\sigma^c}f,S_{\ws}^{1/2}f\right\rangle \\
=&\| S_{\ws}^{-1/2}S_{\ws}^{\sigma^c}f\|^2+(2-\lambda)\left\langle S_{\ws}^{-1/2}S_{\ws}^{\sigma}f, S_{\ws}^{1/2}f\right\rangle +(\lambda-1)\|S_{\ws}^{1/2}f\|^2\\
\ge&(\lambda-\frac{\lambda^2}{4})\|S_{\ws}^{1/2}f\|^2,
\end{align*}
thus
		\begin{align*}
&\left\langle S_{\ws}^{-1}S_{\ws}^{\sigma}f,S_{\ws}^{\sigma}f\right\rangle +\lambda\left\langle S_{\ws}^{\sigma^c}f,f\right\rangle \\
=&\left\langle  S_{\ws}^{-1}S_{\ws}^{\sigma^c}f,S_{\ws}^{\sigma^c}f\right\rangle +(2-\lambda)\left\langle S_{\ws}^{\sigma}f, f\right\rangle +(\lambda-1)\left\langle S_{\ws}f,f\right\rangle \\
\ge&(\lambda-\frac{\lambda^2}{4})\left\langle S_{\ws}f,f\right\rangle .
\end{align*}
Then
		\begin{align*}
&\left\langle S_{\ws}^{-1}S_{\ws}^{\sigma}f,S_{\ws}^{\sigma}f\right\rangle  \\
=&\left\langle  S_{\ws}^{-1}S_{\ws}^{\sigma^c}f,S_{\ws}^{\sigma^c}f\right\rangle +2\left\langle S_{\ws}^{\sigma}f, f\right\rangle -\lambda\left\langle(S_{\ws}^{\sigma}+S_{\ws}^{\sigma^c})f,f\right\rangle   +(\lambda-1)\left\langle S_{\ws}f,f\right\rangle \\
\ge&(\lambda-\frac{\lambda^2}{4})\left\langle S_{\ws}f,f\right\rangle -\lambda\left\langle S_{\ws}^{\sigma^c}f,f\right\rangle,
\end{align*}
thus
		\begin{align}\label{2.40}
\left\langle S_{\ws}^{-1}S_{\ws}^{\sigma}f,S_{\ws}^{\sigma}f\right\rangle 
=&\left\langle  S_{\ws}^{-1}S_{\ws}^{\sigma^c}f,S_{\ws}^{\sigma^c}f\right\rangle +2\left\langle S_{\ws}^{\sigma}f, f\right\rangle -\left\langle S_{\ws}f,f\right\rangle \nonumber\\
\ge&\lambda\left\langle S_{\ws}^{\sigma}f,f\right\rangle -\frac{\lambda^2}{4}\left\langle S_{\ws}f,f\right\rangle,
\end{align}
hence
	\begin{align}\label{2.41}
\left\langle S_{\ws}^{-1}S_{\ws}^{\sigma}f,S_{\ws}^{\sigma}f\right\rangle +\left\langle S_{\ws}^{\sigma^c}f,f\right\rangle 
=&\left\langle  S_{\ws}^{-1}S_{\ws}^{\sigma^c}f,S_{\ws}^{\sigma^c}f\right\rangle +\left\langle S_{\ws}^{\sigma}f,f\right\rangle \nonumber\\
\ge&(\lambda-\frac{\lambda^2}{4})\left\langle S_{\ws}^{\sigma}f,f\right\rangle +(1-\frac{\lambda^2}{4})\left\langle S_{\ws}^{\sigma^c}f,f\right\rangle.
\end{align}
We have
	\begin{align}\label{2.4}
\left\langle S_{\ws}^{-1}S_{\ws}^{\sigma}f,S_{\ws}^{\sigma}f\right\rangle &=\left\langle S_{\ws}S_{\ws}^{-1}S_{\ws}^{\sigma}f,S_{\ws}^{-1}S_{\ws}^{\sigma}f\right\rangle\nonumber\\
&=\left\langle \sum_{i\in \sigma}\left\langle S_{\ws}^{-1}S_{\ws}^{\sigma}f,\phi_i\right\rangle\phi_i+\sum_{i\in \sigma^c}\left\langle S_{\ws}^{-1}S_{\ws}^{\sigma}f,\psi_i\right\rangle\psi_i ,S_{\ws}^{-1}S_{\ws}^{\sigma}f\right\rangle \nonumber\\
&=\left\langle \sum_{i\in \sigma}\left\langle S_{\ws}^{-1}S_{\ws}^{\sigma}f,\phi_i\right\rangle\phi_i,S_{\ws}^{-1}S_{\ws}^{\sigma}f\right\rangle+\left\langle \sum_{i\in \sigma}\left\langle S_{\ws}^{-1}S_{\ws}^{\sigma}f,\psi_i\right\rangle\psi_i,S_{\ws}^{-1}S_{\ws}^{\sigma}f\right\rangle\nonumber\\
&=\sum_{i\in \sigma}\left| \left\langle S_{\ws}^{\sigma}f,S_{\ws}^{-1}\phi_i\right\rangle \right| ^2+\sum_{i\in \sigma^c}\left| \left\langle S_{\ws}^{\sigma}f,S_{\ws}^{-1}\psi_i\right\rangle \right| ^2.
\end{align}
Similarly
\begin{equation}\label{2.5}
\left\langle  S_{\ws}^{-1}S_{\ws}^{\sigma^c}f,S_{\ws}^{\sigma^c}f\right\rangle=\sum_{i\in\sigma}\left| \left\langle S^{\sigma^c}_{\ws}f,S_{\ws}^{-1}\phi_i\right\rangle \right| ^2+\sum_{i\in\sigma^c}\left| \left\langle S^{\sigma^c}_{\ws}f,S_{\ws}^{-1}\psi_i\right\rangle \right| ^2.
\end{equation}

\begin{equation}\label{2.6}
\left\langle S_{\ws}^{\sigma^c}f,f\right\rangle =\sum_{i\in \sigma^c}\left| \left\langle f,\psi_i\right\rangle \right| ^2.
\end{equation}

\begin{equation}\label{2.7}
\left\langle S_{\ws}^{\sigma}f,f\right\rangle =\sum_{i\in \sigma^c}\left| \left\langle f,\phi_i\right\rangle \right| ^2.
\end{equation}
Using equations \eqref{2.41}-\eqref{2.7} in the inequality \eqref{2.3}, we obtain
\begin{align*}
&\sum_{i\in\sigma}\left| \left\langle f,\phi_i\right\rangle \right| ^2+\sum_{i\in\sigma}\left| \left\langle S^{\sigma^c}_{\ws}f,S_{\ws}^{-1}\phi_i\right\rangle \right| ^2+\sum_{i\in\sigma^c}\left| \left\langle S^{\sigma^c}_{\ws}f,S_{\ws}^{-1}\psi_i\right\rangle \right| ^2\\
=&\sum_{i\in\sigma^c}\left| \left\langle f,\psi_i\right\rangle \right| ^2+\sum_{i\in\sigma}\left| \left\langle S^{\sigma}_{\ws}f,S_{\ws}^{-1}\phi_i\right\rangle \right| ^2+\sum_{i\in\sigma^c}\left| \left\langle S^{\sigma}_{\ws}f,S_{\ws}^{-1}\psi_i\right\rangle \right| ^2\\
\ge &(\lambda-\frac{\lambda^2}{4})\sum_{i\in\sigma}\left| \left\langle f,\phi_i\right\rangle \right| ^2+(1-\frac{\lambda^2}{4})\sum_{i\in\sigma^c}\left| \left\langle f,\psi_i\right\rangle \right| ^2.
\end{align*}
\end{proof}
\begin{remark}
If we take $\phi_i=\psi_i$ for all $i\in I$ and $\lambda=1$ in Theorem \ref{thm7.1}, we can obtain Theorem \ref{thm2} with scalar $3/4$.
\end{remark}
\begin{lemma}\label{lem3}
If $P,~Q\in L(\hs)$ satisfy $P+Q=I_{\hs}$.  Then for any $\lambda\in\rs$, we have
\begin{equation*}
P^*P+\lambda(Q^*+Q)=Q^*Q+(1-\lambda)(P^*+P)+(2\lambda-1)I_{\hs}\ge (1-(\lambda-1)^2)I_{\hs}.
\end{equation*}
\end{lemma}
\begin{proof}
	$$P^*P+\lambda(Q^*+Q)=P^*P+\lambda(I_{\hs}-P^*+I_P)=P^*P-\lambda(P^*+P)+2\lambda I_{\hs},$$
	and
	\begin{align*}
Q^*Q+(1-\lambda)(P^*+P)+(2\lambda-1)I_{\hs}
=&(I_{\hs}-P^*)(I_{\hs}-P)+(1-\lambda)(P^*+P)+(2\lambda-1)I_{\hs}\\
=&P^*P-\lambda(P^*+P)+2\lambda I_{\hs}\\
=&(P-\lambda I_{\hs})^*(P-\lambda I_{\hs})+(1-(\lambda-1)^2)I_{\hs}\\
\ge& (1-(\lambda-1)^2)I_{\hs}.
\end{align*}	
Hence the result follows.
\end{proof}
\begin{theorem}\label{thm3}
	Suppose two frames $\{\phi_i\}_{i\in I}$ and $\{\psi_i\}_{i\in I}$  for a Hilbert space $\hs$ are woven and $\{\varphi_i\}_{i\in I}$  is an alternate dual frame of the weaving frame $\{\phi_i\}_{i\in\sigma}\cup\{\psi_i\}_{i\in\sigma^c}$. Then for any $\lambda\in\rs$, for all $\sigma\subset I$ and all $f\in\hs$, we have
	\begin{align}\label{2.8}
	&{\rm Re}\left(  \sum_{i\in\sigma}\left\langle f,\varphi_i\right\rangle \overline{\left\langle f,\phi_i\right\rangle } \right)+\left\| \sum_{i\in\sigma^c}\left\langle f,\varphi_i\right\rangle \psi_i\right\| ^2  \nonumber\\
	=&{\rm Re}\left(  \sum_{i\in\sigma^c}\left\langle f,\varphi_i\right\rangle \overline{\left\langle f,\psi_i\right\rangle } \right)+\left\| \sum_{i\in\sigma}\left\langle f,\varphi_i\right\rangle \phi_i\right\| ^2\nonumber\\
	\ge &(2\lambda-\lambda^2){\rm Re}\left(  \sum_{i\in\sigma}\left\langle f,\varphi_i\right\rangle \overline{\left\langle f,\phi_i\right\rangle } \right)+(1-\lambda^2){\rm Re}\left(  \sum_{i\in\sigma^c}\left\langle f,\varphi_i\right\rangle \overline{\left\langle f,\psi_i\right\rangle } \right).
	\end{align}
\end{theorem}
\begin{proof}
	For all $f\in\hs$ and all $\sigma\subset I$, define the operators 
	$$E_{\sigma}f=\sum_{i\in\sigma}\left\langle f,\varphi_i\right\rangle \phi_i,~~~~E_{\sigma^c}f=\sum_{i\in\sigma^c}\left\langle f,\varphi_i\right\rangle \psi_i.$$
	Then the series converge unconditionally and $E_{\sigma},E_{\sigma^c}\in L(\hs)$.  By \eqref{1.1}, we have $E_{\sigma}+E_{\sigma^c}=I_{\hs}$. Applying Lemma \ref{lem3} to the operators $P=E_{\sigma}$ and $Q=E_{\sigma^c}$, for all $f\in\hs$, we obtain
	\begin{align}
	&\quad\left\langle E_{\sigma}^*E_{\sigma}f,f\right\rangle +\lambda\left\langle (E_{\sigma^c}^*+E_{\sigma^c})f,f\right\rangle\nonumber \\
	&=\left\langle E_{\sigma}^*E_{\sigma}f,f\right\rangle +\lambda\overline{\left\langle E_{\sigma^c}f,f\right\rangle}+\lambda\left\langle E_{\sigma}f,f\right\rangle\label{2.9}\\
	&=\left\langle E_{\sigma^c}^*E_{\sigma^c}f,f\right\rangle  +(1-\lambda)\left\langle (E_{\sigma}^*+E_{\sigma})f,f\right\rangle+(2\lambda-1)\|f\|^2 \nonumber\\
	&=\left\langle E_{\sigma^c}^*E_{\sigma^c}f,f\right\rangle+(1-\lambda)(\overline{\left\langle E_{\sigma}f,f\right\rangle }+\left\langle E_{\sigma}f,f\right\rangle)+(2\lambda-1)\left\langle I_{\hs}f,f\right\rangle .\label{2.10}
	\end{align}
	A simple computation of  \eqref{2.9} and \eqref{2.10}, we have
	$$\|E_{\sigma}f\|^2+2\lambda{\rm Re}\left\langle E_{\sigma^c}f,f\right\rangle =\|E_{\sigma^c}f\|^2+2(1-\lambda){\rm Re}\left\langle E_{\sigma}f,f\right\rangle+(2\lambda-1){\rm Re}\left\langle I_{\hs}f,f\right\rangle.$$
	Then,
	\begin{align*}
	\|E_{\sigma}f\|^2&=\|E_{\sigma^c}f\|^2+2(1-\lambda){\rm Re}\left\langle E_{\sigma}f,f\right\rangle-2\lambda{\rm Re}\left\langle E_{\sigma^c}f,f\right\rangle+(2\lambda-1){\rm Re}\left\langle I_{\hs}f,f\right\rangle\\
	&=\|E_{\sigma^c}f\|^2+2{\rm Re}\left\langle E_{\sigma}f,f\right\rangle-2\lambda{\rm Re}\left\langle (E_{\sigma}+E_{\sigma^c})f,f\right\rangle+(2\lambda-1){\rm Re}\left\langle I_{\hs}f,f\right\rangle\\
	&=\|E_{\sigma^c}f\|^2+2{\rm Re}\left\langle E_{\sigma}f,f\right\rangle-{\rm Re}\left\langle I_{\hs}f,f\right\rangle\\
	&=\|E_{\sigma^c}f\|^2+2{\rm Re}\left\langle E_{\sigma}f,f\right\rangle-{\rm Re}\left\langle (E_{\sigma}+E_{\sigma^c})f,f\right\rangle\\
	&=\|E_{\sigma^c}f\|^2+{\rm Re}\left\langle E_{\sigma}f,f\right\rangle-{\rm Re}\left\langle E_{\sigma^c}f,f\right\rangle.
	\end{align*}
	Hence,
	\begin{equation}\label{2.11}
\|E_{\sigma}f\|^2+{\rm Re}\left\langle E_{\sigma^c}f,f\right\rangle=\|E_{\sigma^c}f\|^2+{\rm Re}\left\langle E_{\sigma}f,f\right\rangle.
	\end{equation}
	Since,
	\begin{equation}\label{2.12}
\|E_{\sigma}f\|^2=\left\| \sum_{i\in \sigma}\left\langle f,\varphi_i\right\rangle \phi_i\right\| ^2.
	\end{equation}
	\begin{equation}\label{2.13}
{\rm Re}\left\langle E_{\sigma^c}f,f\right\rangle={\rm Re}\left(  \sum_{i\in\sigma^c}\left\langle f,\varphi_i\right\rangle \overline{\left\langle f,\psi_i\right\rangle } \right).
\end{equation}
	\begin{equation}\label{2.14}
\|E_{\sigma^c}f\|^2=\left\| \sum_{i\in \sigma}\left\langle f,\varphi_i\right\rangle \psi_i\right\| ^2.
\end{equation}
\begin{equation}\label{2.15}
{\rm Re}\left\langle E_{\sigma}f,f\right\rangle={\rm Re}\left(  \sum_{i\in\sigma}\left\langle f,\varphi_i\right\rangle \overline{\left\langle f,\phi_i\right\rangle } \right).
\end{equation}
Using equations \eqref{2.11}-\eqref{2.15}, we have
\begin{align*}
{\rm Re}\left(  \sum_{i\in\sigma}\left\langle f,\varphi_i\right\rangle \overline{\left\langle f,\phi_i\right\rangle } \right)+\left\| \sum_{i\in\sigma^c}\left\langle f,\varphi_i\right\rangle \psi_i\right\| ^2
={\rm Re}\left(  \sum_{i\in\sigma^c}\left\langle f,\varphi_i\right\rangle \overline{\left\langle f,\psi_i\right\rangle } \right)+\left\| \sum_{i\in\sigma}\left\langle f,\varphi_i\right\rangle \phi_i\right\| ^2.
\end{align*}
	We now prove the inequality of \eqref{2.8}. 	From Lemma \ref{lem3}, we have
	\begin{equation}\label{2.16}
\left\langle E_{\sigma}^*E_{\sigma}f,f\right\rangle +\lambda\overline{\left\langle E_{\sigma^c}f,f\right\rangle}+\lambda\left\langle E_{\sigma^c}f,f\right\rangle\ge (2\lambda-\lambda^2)\left\langle I_{\hs}f,f\right\rangle.
	\end{equation}
Then
	$$\|E_{\sigma}f\|^2+2\lambda{\rm Re}\left\langle E_{\sigma^c}f,f\right\rangle \ge (2\lambda-\lambda^2){\rm Re}\left\langle I_{\hs}f,f\right\rangle,$$
	hence
	\begin{align*}
	\|E_{\sigma}f\|^2&\ge (2\lambda-\lambda^2){\rm Re}\left\langle I_{\hs}f,f\right\rangle-2\lambda{\rm Re}\left\langle E_{\sigma^c}f,f\right\rangle\\
	&=(2\lambda-\lambda^2){\rm Re}\left\langle (E_{\sigma}+E_{\sigma^c})f,f\right\rangle-2\lambda{\rm Re}\left\langle E_{\sigma^c}f,f\right\rangle\\
	&=(2\lambda-\lambda^2){\rm Re}\left\langle E_{\sigma}f,f\right\rangle -\lambda^2{\rm Re}\left\langle E_{\sigma^c}f,f\right\rangle \\
	&=(2\lambda-\lambda^2){\rm Re}\left\langle E_{\sigma}f,f\right\rangle+(1-\lambda^2){\rm Re}\left\langle E_{\sigma^c}f,f\right\rangle -{\rm Re}\left\langle E_{\sigma^c}f,f\right\rangle.
	\end{align*}
	Therefore,
	\begin{equation}\label{2.17}
	\|E_{\sigma}f\|^2+{\rm Re}\left\langle E_{\sigma^c}f,f\right\rangle\ge (2\lambda-\lambda^2){\rm Re}\left\langle E_{\sigma}f,f\right\rangle+(1-\lambda^2){\rm Re}\left\langle E_{\sigma^c}f,f\right\rangle.
	\end{equation}
Using equations \eqref{2.12}-\eqref{2.15} and \eqref{2.17}, we have	
$${\rm Re}\left(  \sum_{i\in\sigma}\left\langle f,\varphi_i\right\rangle \overline{\left\langle f,\phi_i\right\rangle } \right)+\left\| \sum_{i\in\sigma^c}\left\langle f,\varphi_i\right\rangle \psi_i\right\| ^2\ge (2\lambda-\lambda^2){\rm Re}\left(  \sum_{i\in\sigma}\left\langle f,\varphi_i\right\rangle \overline{\left\langle f,\phi_i\right\rangle } \right)+(1-\lambda^2){\rm Re}\left(  \sum_{i\in\sigma^c}\left\langle f,\varphi_i\right\rangle \overline{\left\langle f,\psi_i\right\rangle } \right).$$
The proof is completed.
\end{proof}
\begin{remark}
Theorem \ref{thm2.1} can be obtained from Theorem \ref{thm3} by taking $\phi_i=\psi_i$ for all $i\in I$ and $\lambda=\frac{1}{2}$.
\end{remark}
\begin{theorem}\label{thm4}
	Suppose $\Phi=\{\phi_i\}_{i\in I}$ and $\Psi=\{\psi_i\}_{i\in I}$  for a Hilbert space $\hs$ are woven and $\{\varphi_i\}_{i\in I}$  is an alternate dual frame of the weaving frame $\{\phi_i\}_{i\in\sigma}\cup\{\phi_i\}_{i\in\sigma^c}$. Then for any $\lambda\in\rs$, for all $\sigma\subset I$ and all $f\in\hs$, we have
	\begin{align}\label{2.18}
	\left(  \sum_{i\in\sigma}\left\langle f,\varphi_i\right\rangle \overline{\left\langle f,\phi_i\right\rangle}  \right)+\left\| \sum_{i\in\sigma^c}\left\langle f,\varphi_i\right\rangle \psi_i\right\| ^2
	=\overline{\left(  \sum_{i\in\sigma^c}\left\langle f,\varphi_i\right\rangle \overline{\left\langle f,\psi_i\right\rangle } \right)}+\left\| \sum_{i\in\sigma}\left\langle f,\varphi\right\rangle \phi_i\right\| ^2.	 
	\end{align}
\end{theorem}
\begin{proof}
For $\sigma\subset I$ and $f\in\hs$, we define the operator $E_{\sigma}$ and $E_{\sigma^c}$ as in Theorem \ref{thm3}. Therefore, we have $E_{\sigma}+E_{\sigma^c}=I_{\hs}$. By Lemma \ref{lem1}, we have
	\begin{align*}
	\left(  \sum_{i\in\sigma}\left\langle f,\varphi_i\right\rangle \overline{\left\langle f,\phi_i\right\rangle}  \right)+\left\| \sum_{i\in\sigma^c}\left\langle f,\varphi_i\right\rangle \psi_i\right\| ^2
&=\left\langle E_{\sigma}f,f\right\rangle +\left\langle E_{\sigma^c}^*E_{\sigma^c}f,f\right\rangle \\
&=\left\langle E_{\sigma^c}^*f,f\right\rangle +\left\langle E_{\sigma}^*E_{\sigma}f,f\right\rangle\\
&=\overline{\left\langle E_{\sigma^c}^*f,f\right\rangle}	 +\|E_{\sigma}f\|^2\\
&=\overline{\left(  \sum_{i\in\sigma^c}\left\langle f,\varphi_i\right\rangle \overline{\left\langle f,\psi_i\right\rangle } \right)}+\left\| \sum_{i\in\sigma}\left\langle f,\varphi\right\rangle \phi_i\right\| ^2.
\end{align*}
Hence \eqref{2.18} holds. The proof is completed.
\end{proof}
\begin{theorem}\label{thm5}
	Suppose two frames $\{\phi_i\}_{i\in I}$ and $\{\psi_i\}_{i\in I}$  for a Hilbert space $\hs$ are woven and $\{\varphi_i\}_{i\in I}$  is an alternate dual frame of the weaving frame $\{\phi_i\}_{i\in\sigma}\cup\{\phi_i\}_{i\in\sigma^c}$.  Then for every bounded sequence $\{a_i\}_{i\in I}$ and every $f\in\hs$, we have
\begin{align*}
&\left(  \sum_{i\in\sigma}a_i\left\langle f,\varphi_i\right\rangle \overline{\left\langle f,\phi_i\right\rangle}  \right)+\left(  \sum_{i\in\sigma^c}a_i\left\langle f,\varphi_i\right\rangle \overline{\left\langle f,\psi_i\right\rangle}  \right)+\left\| \sum_{i\in\sigma^c}(1-a_i)\left\langle f,\varphi_i\right\rangle \psi_i+ \sum_{i\in\sigma}(1-a_i)\left\langle f,\varphi_i\right\rangle \phi_i\right\| ^2\\
=&\left\| \sum_{i\in\sigma}a_i\left\langle f,\varphi_i\right\rangle \phi_i+ \sum_{i\in\sigma^c}a_i\left\langle f,\varphi_i\right\rangle \psi_i\right\| ^2+\overline{\left(  \sum_{i\in\sigma}(1-a_i)\left\langle f,\varphi_i\right\rangle \overline{\left\langle f,\phi_i\right\rangle } \right)}+\overline{\left(  \sum_{i\in\sigma^c}(1-a_i)\left\langle f,\varphi_i\right\rangle \overline{\left\langle f,\psi_i\right\rangle } \right)}
.	 
\end{align*}	
\end{theorem}
\begin{proof}
For all $\sigma\subset I$ and $f\in\hs$, we define the operators 
$$E_{\sigma}f=\sum_{i\in \sigma}a_i\left\langle f,\varphi_i\right\rangle \phi_i, ~~E_{\sigma^c}f=\sum_{i\in \sigma^c}a_i\left\langle f,\varphi_i\right\rangle \psi_i,$$ 
and 
$$F_{\sigma}f=\sum_{i\in \sigma}(1-a_i)\left\langle f,\varphi_i\right\rangle \phi_i,~~F_{\sigma^c}f=\sum_{i\in \sigma}(1-a_i)\left\langle f,\varphi_i\right\rangle \psi_i.$$  Note that these series converge unconditionally. Also we have $E_{\sigma},E_{\sigma^c},F_{\sigma}, F_{\sigma^c}\in L(\hs)$ and $E_{\sigma}+E_{\sigma^c}+F_{\sigma}+ F_{\sigma^c}=I_{\hs}$. Applying Lemma \ref{lem1} to the operators $P=E_{\sigma}+E_{\sigma^c}$ and $Q=F_{\sigma}+F_{\sigma^c}$ and for every $f\in\hs$, we have
\begin{align*}
&\left(  \sum_{i\in\sigma}a_i\left\langle f,\varphi_i\right\rangle \overline{\left\langle f,\phi_i\right\rangle}  \right)+\left(  \sum_{i\in\sigma^c}a_i\left\langle f,\varphi_i\right\rangle \overline{\left\langle f,\psi_i\right\rangle}  \right)+\left\| \sum_{i\in\sigma^c}(1-a_i)\left\langle f,\varphi_i\right\rangle \psi_i+ \sum_{i\in\sigma}(1-a_i)\left\langle f,\varphi_i\right\rangle \phi_i\right\| ^2\\
=&\left\langle E_{\sigma}f,f\right\rangle +\left\langle E_{\sigma^c}f,f\right\rangle+\left\langle (F_{\sigma}+F_{\sigma^c})^*(F_{\sigma}+F_{\sigma^c})f,f\right\rangle \\
=&\left\langle (E_{\sigma}+E_{\sigma^c})f,f\right\rangle +\left\langle (F_{\sigma}+F_{\sigma^c})^*(F_{\sigma}+F_{\sigma^c})f,f\right\rangle \\
=&\left\langle (F_{\sigma}+F_{\sigma^c})^*f,f\right\rangle +\left\langle (E_{\sigma}+E_{\sigma^c})^*(E_{\sigma}+E_{\sigma^c})f,f\right\rangle \\
=&\overline{\left\langle (F_{\sigma}+F_{\sigma^c})f,f\right\rangle }+\|(E_{\sigma}+E_{\sigma^c})f\|^2\\
=&\|(E_{\sigma}+E_{\sigma^c})f\|^2+\overline{\left\langle F_{\sigma}f,f\right\rangle }+\overline{\left\langle F_{\sigma^c}f,f\right\rangle }\\
=&\left\| \sum_{i\in\sigma}a_i\left\langle f,\varphi_i\right\rangle \phi_i+ \sum_{i\in\sigma^c}a_i\left\langle f,\varphi_i\right\rangle \psi_i\right\| ^2+\overline{\left(  \sum_{i\in\sigma}(1-a_i)\left\langle f,\varphi_i\right\rangle \overline{\left\langle f,\phi_i\right\rangle } \right)}+\overline{\left(  \sum_{i\in\sigma^c}(1-a_i)\left\langle f,\varphi_i\right\rangle \overline{\left\langle f,\psi_i\right\rangle } \right)}.
\end{align*}	
Hence the relation holds.
\end{proof}
Observe that if we consider $\sigma\subset I$ and
$$
a_i=\Bigg\{ 
\begin{array}{c}
0,~{\rm if}~ i\in\sigma \\ 
1,~{\rm if} ~i\in\sigma^c
\end{array},
$$
then Theorem \ref{thm4} follows from Theorem \ref{thm5}.
\begin{remark}
If we take $\phi_i=\psi_i$ for all $i\in I$ in Theorem \ref{thm4} and Theorem \ref{thm5}, we can obtain Theorem \ref{thm2.2} and Theorem 2.3 of \cite{zhu2010note}.
\end{remark}
\begin{theorem}\label{thm6}
Suppose two frames $\{\phi_i\}_{i\in I}$ and $\{\psi_i\}_{i\in I}$  for a Hilbert space $\hs$ are woven. Then for any $\lambda\in\rs$, $\sigma\subset I$ and $f\in\hs$, we have
\begin{align*}
0\le &\sum_{i\in \sigma}\left| \left\langle f,\phi_i\right\rangle \right| ^2-\sum_{i\in \sigma}\left| \left\langle S_{\ws}^{\sigma}f,S_{\ws}^{-1}\phi_i \right\rangle \right|^2 -\sum_{i\in \sigma^c}\left| \left\langle S_{\ws}^{\sigma}f,S_{\ws}^{-1}\psi_i \right\rangle \right|^2\\
\le & \frac{\lambda^2}{4}\sum_{i\in \sigma^c}\left| \left\langle f,\psi_i\right\rangle \right| ^2+(1-\frac{\lambda}{2})^2 \sum_{i\in \sigma}\left| \left\langle f,\phi_i\right\rangle \right| ^2.
\end{align*}
\end{theorem}
\begin{proof}
Considering positive operators $P=S_{\ws}^{-1/2}S_{\ws}^{\sigma}S_{\ws}^{-1/2}$, $Q=S_{\ws}^{-1/2}S_{\ws}^{\sigma^c}S_{\ws}^{-1/2}$, then $P+Q=I_{\hs}$, and 
$$PQ=P(I_{\hs}-P)=P-P^2=(I_{\hs}-P)P=QP,$$
then 
$$0\le PQ=P(I_{\hs}-P)=P-P^2=S_{\ws}^{-1/2}(S_{\ws}^{\sigma}-S_{\ws}^{\sigma}S_{\ws}^{-1}S_{\ws}^{\sigma})S_{\ws}^{-1/2},$$
which follows $S_{\ws}^{\sigma}-S_{\ws}^{\sigma}S_{\ws}^{-1}S_{\ws}^{\sigma}\ge 0$. Then for all $f\in\hs$, we have
\begin{align*}
&\sum_{i\in \sigma}\left| \left\langle f,\phi_i\right\rangle \right| ^2-\sum_{i\in \sigma}\left| \left\langle S_{\ws}^{\sigma}f,S_{\ws}^{-1}\phi_i \right\rangle \right|^2 -\sum_{i\in \sigma^c}\left| \left\langle S_{\ws}^{\sigma}f,S_{\ws}^{-1}\psi_i \right\rangle \right|^2\\
= & \left\langle S_{\ws}^{\sigma}f,f\right\rangle-\left\langle S_{\ws}^{-1}S_{\ws}^{\sigma}f,S_{\ws}^{\sigma}f\right\rangle\\
= & \left\langle (S_{\ws}^{\sigma}-S_{\ws}^{\sigma}S_{\ws}^{-1}S_{\ws}^{\sigma})f,f\right\rangle
\ge 0
\end{align*}
 By \eqref{2.40}, we have
 	\begin{align*}
 \left\langle S_{\ws}^{-1}S_{\ws}^{\sigma}f,S_{\ws}^{\sigma}f\right\rangle 
-\left\langle S_{\ws}^{\sigma}f, f\right\rangle 
 \ge&\lambda\left\langle S_{\ws}^{\sigma}f,f\right\rangle -\frac{\lambda^2}{4}\left\langle S_{\ws}f,f\right\rangle-\left\langle S_{\ws}^{\sigma}f, f\right\rangle,
 \end{align*}
and then
 	\begin{align*}
&\sum_{i\in \sigma}\left| \left\langle f,\phi_i\right\rangle \right| ^2-\sum_{i\in \sigma}\left| \left\langle S_{\ws}^{\sigma}f,S_{\ws}^{-1}\phi_i \right\rangle \right|^2 -\sum_{i\in \sigma^c}\left| \left\langle S_{\ws}^{\sigma}f,S_{\ws}^{-1}\psi_i \right\rangle \right|^2\\
&=\left\langle S_{\ws}^{\sigma}f,f\right\rangle-\left\langle S_{\ws}^{-1}S_{\ws}^{\sigma}f,S_{\ws}^{\sigma}f\right\rangle\\
&\le \left\langle S_{\ws}^{\sigma}f, f\right\rangle-\lambda\left\langle S_{\ws}^{\sigma}f,f\right\rangle +\frac{\lambda^2}{4}\left\langle S_{\ws}f,f\right\rangle\\
&=(1-\lambda)\left\langle S_{\ws}^{\sigma}f, f\right\rangle+\frac{\lambda^2}{4}\left\langle S_{\ws}f,f\right\rangle\\
&=(1-\lambda)\left\langle (S_{\ws}-S_{\ws}^{\sigma^c})f, f\right\rangle+\frac{\lambda^2}{4}\left\langle S_{\ws}f,f\right\rangle
\\
&=(\lambda-1)\left\langle S_{\ws}^{\sigma^c}f, f\right\rangle+(1-\frac{\lambda}{2})^2\left\langle S_{\ws}f,f\right\rangle \\
&=\frac{\lambda^2}{4}\left\langle S_{\ws}^{\sigma^c}f, f\right\rangle+(1-\frac{\lambda}{2})^2\left\langle S_{\ws}^{\sigma}f,f\right\rangle \\
&=\frac{\lambda^2}{4}\sum_{i\in \sigma^c}\left| \left\langle f,\psi_i\right\rangle \right| ^2+(1-\frac{\lambda}{2})^2 \sum_{i\in \sigma}\left| \left\langle f,\phi_i\right\rangle \right| ^2.
\end{align*} 
\end{proof}
\begin{theorem}\label{thm7}
	Suppose two frames $\{\phi_i\}_{i\in I}$ and $\{\psi_i\}_{i\in I}$  for a Hilbert space $\hs$ are woven. Then for any $\lambda\in\rs$, $\sigma\subset I$ and $f\in\hs$, we have
	\begin{align*}
&(2\lambda-\frac{\lambda^2}{2}-1)\sum_{i\in \sigma}\left| \left\langle f,\phi_i\right\rangle \right| ^2+(1-\frac{\lambda^2}{2})\sum_{i\in \sigma^c}\left| \left\langle f,\psi_i\right\rangle \right| ^2\\
	\le &\sum_{i\in \sigma}\left| \left\langle S_{\ws}^{\sigma}f,S_{\ws}^{-1}\phi_i\right\rangle \right| ^2 +\sum_{i\in \sigma^c}\left| \left\langle S_{\ws}^{\sigma}f,S_{\ws}^{-1}\psi_i\right\rangle \right| ^2+\sum_{i\in \sigma}\left| \left\langle S_{\ws}^{\sigma^c}f,S_{\ws}^{-1}\phi_i\right\rangle \right| ^2 +\sum_{i\in \sigma^c}\left| \left\langle S_{\ws}^{\sigma^c}f,S_{\ws}^{-1}\psi_i\right\rangle \right| ^2\\
	 \le & \sum_{i\in \sigma}\left| \left\langle f,\phi_i\right\rangle \right| ^2+\sum_{i\in \sigma^c}\left| \left\langle f,\psi_i\right\rangle \right| ^2.
	\end{align*}
\end{theorem}
\begin{proof}
By \eqref{2.41}, we have 
	\begin{align}\label{2.21}
\left\langle S_{\ws}^{-1}S_{\ws}^{\sigma}f,S_{\ws}^{\sigma}f\right\rangle 
\ge&(\lambda-\frac{\lambda^2}{4})\left\langle S_{\ws}^{\sigma}f,f\right\rangle -\frac{\lambda^2}{4}\left\langle S_{\ws}^{\sigma^c}f,f\right\rangle.
\end{align}
	\begin{align}\label{2.22}
\left\langle  S_{\ws}^{-1}S_{\ws}^{\sigma^c}f,S_{\ws}^{\sigma^c}f\right\rangle 
\ge&(\lambda-\frac{\lambda^2}{4}-1)\left\langle S_{\ws}^{\sigma}f,f\right\rangle +(1-\frac{\lambda^2}{4})\left\langle S_{\ws}^{\sigma^c}f,f\right\rangle.
\end{align}
From \eqref{2.21} and  \eqref{2.21}, we obtain
	\begin{align*}
&\sum_{i\in \sigma}\left| \left\langle S_{\ws}^{\sigma}f,S_{\ws}^{-1}\phi_i\right\rangle \right| ^2 +\sum_{i\in \sigma^c}\left| \left\langle S_{\ws}^{\sigma}f,S_{\ws}^{-1}\psi_i\right\rangle \right| ^2+\sum_{i\in \sigma}\left| \left\langle S_{\ws}^{\sigma^c}f,S_{\ws}^{-1}\phi_i\right\rangle \right| ^2 +\sum_{i\in \sigma^c}\left| \left\langle S_{\ws}^{\sigma^c}f,S_{\ws}^{-1}\psi_i\right\rangle \right| ^2\\
= & \left\langle S_{\ws}^{-1}S_{\ws}^{\sigma}f,S_{\ws}^{\sigma}f\right\rangle +\left\langle  S_{\ws}^{-1}S_{\ws}^{\sigma^c}f,S_{\ws}^{\sigma^c}f\right\rangle \\
\ge &(2\lambda-\frac{\lambda^2}{2}-1)\left\langle S_{\ws}^{\sigma}f,f\right\rangle +(1-\frac{\lambda^2}{2})\left\langle S_{\ws}^{\sigma^c}f,f\right\rangle\\
=&(2\lambda-\frac{\lambda^2}{2}-1)\sum_{i\in \sigma}\left| \left\langle f,\phi_i\right\rangle \right| ^2+(1-\frac{\lambda^2}{2})\sum_{i\in \sigma^c}\left| \left\langle f,\psi_i\right\rangle \right| ^2.
\end{align*}
Next, we prove the last part. Let $P=S_{\ws}^{-1/2}S_{\ws}^{\sigma}S_{\ws}^{-1/2}$, $Q=S_{\ws}^{-1/2}S_{\ws}^{\sigma^c}S_{\ws}^{-1/2}$. Since $PQ=QP$, we have 
$$P-P^2=P(I_{\hs}-P)=PQ\ge 0,$$
then for all $f\in\hs$, $\|Pf\|^2\le \left\langle Pf,f\right\rangle $. Similarly, $\|Qf\|^2\le \left\langle Qf,f\right\rangle $.
Hence,
\begin{align*}
&\sum_{i\in \sigma}\left| \left\langle S_{\ws}^{\sigma}f,S_{\ws}^{-1}\phi_i\right\rangle \right| ^2 +\sum_{i\in \sigma^c}\left| \left\langle S_{\ws}^{\sigma}f,S_{\ws}^{-1}\psi_i\right\rangle \right| ^2+\sum_{i\in \sigma}\left| \left\langle S_{\ws}^{\sigma^c}f,S_{\ws}^{-1}\phi_i\right\rangle \right| ^2 +\sum_{i\in \sigma^c}\left| \left\langle S_{\ws}^{\sigma^c}f,S_{\ws}^{-1}\psi_i\right\rangle \right| ^2\\
= & \left\langle S_{\ws}^{-1}S_{\ws}^{\sigma}f,S_{\ws}^{\sigma}f\right\rangle +\left\langle  S_{\ws}^{-1}S_{\ws}^{\sigma^c}f,S_{\ws}^{\sigma^c}f\right\rangle \\
= & \left\langle S_{\ws}^{-1/2}S_{\ws}^{\sigma}f,S_{\ws}^{-1/2}S_{\ws}^{\sigma}f\right\rangle +\left\langle  S_{\ws}^{-1/2}S_{\ws}^{\sigma^c}f,S_{\ws}^{-1/2}S_{\ws}^{\sigma^c}f\right\rangle \\
= & \left\langle S_{\ws}^{-1/2}S_{\ws}^{\sigma}S_{\ws}^{-1/2}S_{\ws}^{1/2}f,S_{\ws}^{-1/2}S_{\ws}^{\sigma}S_{\ws}^{-1/2}S_{\ws}^{1/2}f\right\rangle +\left\langle  S_{\ws}^{-1/2}S_{\ws}^{\sigma^c}S_{\ws}^{-1/2}S_{\ws}^{1/2}f,S_{\ws}^{-1/2}S_{\ws}^{\sigma^c}S_{\ws}^{-1/2}S_{\ws}^{1/2}f\right\rangle \\
\le &\left\langle S_{\ws}^{-1/2}S_{\ws}^{\sigma}S_{\ws}^{-1/2}S_{\ws}^{1/2}f,S_{\ws}^{1/2}f\right\rangle +\left\langle  S_{\ws}^{-1/2}S_{\ws}^{\sigma^c}S_{\ws}^{-1/2}S_{\ws}^{1/2}f,S_{\ws}^{1/2}f\right\rangle\\
= &\left\langle S_{\ws}^{\sigma}f,f\right\rangle +\left\langle S_{\ws}^{\sigma^c}f,f\right\rangle\\
=&\sum_{i\in \sigma}\left| \left\langle f,\phi_i\right\rangle \right| ^2+\sum_{i\in \sigma^c}\left| \left\langle f,\psi_i\right\rangle \right| ^2.
\end{align*}
\end{proof}
By  Theorem \ref{thm6} and Theorem \ref{thm7} , we immediately get the following results.
\begin{corollary}
Suppose two frames $\{\phi_i\}_{i\in I}$ and $\{\psi_i\}_{i\in I}$  for a Hilbert space $\hs$ are $A$-woven. Then for any $\lambda\in\rs$, $\sigma\subset I$ and $f\in\hs$, we have
\begin{align*}
0\le A\sum_{i\in \sigma}\left| \left\langle f,\phi_i\right\rangle \right| ^2-\left\| \sum_{i\in \sigma}\left\langle f,\phi_i\right\rangle \phi_i\right\| ^2
\le  \frac{A\lambda^2}{4}\sum_{i\in \sigma^c}\left| \left\langle f,\psi_i\right\rangle \right| ^2+(1-\frac{\lambda}{2})^2A \sum_{i\in \sigma}\left| \left\langle f,\phi_i\right\rangle \right| ^2
\end{align*}
and
\begin{align*}
(2\lambda-\frac{\lambda^2}{2}-1)A\sum_{i\in \sigma}\left| \left\langle f,\phi_i\right\rangle \right| ^2+(1-\frac{\lambda^2}{2})A\sum_{i\in \sigma^c}\left| \left\langle f,\psi_i\right\rangle \right| ^2
\le \left\| \sum_{i\in \sigma}\left\langle f,\phi_i\right\rangle \phi_i\right\| ^2+\left\| \sum_{i\in \sigma^c}\left\langle f,\psi_i\right\rangle \psi_i\right\| ^2
\le A \|f\|^2.
\end{align*}
\begin{proof}
Since $\{\phi_i\}_{i\in I}$ and $\{\psi_i\}_{i\in I}$   are $A$-woven, we have $S_{\ws}^{-1}=\frac{1}{A}I_{\hs}$, and then the results hold by Theorem \ref{thm6} and Theorem \ref{thm7}.
\end{proof}
\end{corollary}
\begin{remark}
If we take $\lambda=1$ and $\phi_i=\psi_i$ for all $i\in I$ in Theorem \ref{thm6} and Theorem \ref{thm7}, we obtain the similar inequalities in Theorem 5 and Theorem 6 of \cite{li2017somee}.
\end{remark}
\section*{Acknowledgements}
The research is supported by the National Natural Science Foundation of China (LJT10110010115). 
\bibliographystyle{plain}
\bibliography{my}

\begin{thebibliography}{10}

\bibitem{balan2006signal}
R.~Balan, P.~Casazza, and D.~Edidin.
\newblock On signal reconstruction without phase.
\newblock {\em Applied and Computational Harmonic Analysis}, 20(3):345--356,
  2006.

\bibitem{balan2005decompositions}
R.~Balan, P.~Casazza, D.~Edidin, and G.~Kutyniok.
\newblock Decompositions of frames and a new frame identity.
\newblock In {\em Proceedings of SPIE}, volume 5914, pages 1--10, 2005.

\bibitem{balan2007new}
R.~Balan, P.~Casazza, D.~Edidin, and G.~Kutyniok.
\newblock A new identity for \protect{P}arseval frames.
\newblock {\em Proceedings of the American Mathematical Society},
  135(4):1007--1015, 2007.

\bibitem{bemrose2015weaving}
Travis Bemrose, Peter~G Casazza, Karlheinz Gr{\"o}chenig, Mark~C Lammers, and
  Richard~G Lynch.
\newblock Weaving frames.
\newblock {\em arXiv preprint arXiv:1503.03947}, 2015.

\bibitem{casazza2016weaving}
Peter~G Casazza, Daniel Freeman, and Richard~G Lynch.
\newblock Weaving \protect{S}chauder frames.
\newblock {\em Journal of Approximation Theory}, 211:42--60, 2016.

\bibitem{daubechies1986painless}
Ingrid Daubechies, Alex Grossmann, and Yves Meyer.
\newblock Painless nonorthogonal expansions.
\newblock {\em Journal of Mathematical Physics}, 27(5):1271--1283, 1986.

\bibitem{donoho2005continuous}
David Donoho and E~Candes.
\newblock Continuous curvelet transform: \protect{II}. discretization and
  frames.
\newblock {\em Applied and Computational Harmonic Analysis}, 19(2):198--222,
  2005.

\bibitem{duffin1952class}
Richard~J Duffin and Albert~C Schaeffer.
\newblock A class of nonharmonic \protect{F}ourier series.
\newblock {\em Transactions of the American Mathematical Society},
  72(2):341--366, 1952.

\bibitem{eldar2002optimal}
Yonina~C Eldar and G~David Forney.
\newblock Optimal tight frames and quantum measurement.
\newblock {\em IEEE Transactions on Information Theory}, 48(3):599--610, 2002.

\bibitem{guavructa2006some}
P.~G{\u{a}}vru{\c{t}}a.
\newblock On some identities and inequalities for frames in \protect{H}ilbert
  spaces.
\newblock {\em Journal of mathematical analysis and applications},
  321(1):469--478, 2006.

\bibitem{han2000frames}
Deguang Han and David~R Larson.
\newblock {\em Frames, bases and group representations}, volume 697.
\newblock American Mathematical Soc., 2000.

\bibitem{khosravi2018weaving}
Amir Khosravi and Jamaleh~Sohrabi Banyarani.
\newblock Weaving g-frames and weaving fusion frames.
\newblock {\em Bulletin of the Malaysian Mathematical Sciences Society}, pages
  1--19, 2018.

\bibitem{kovacevic2002filter}
Jelena Kovacevic, Pier~Luigi Dragotti, and Vivek~K Goyal.
\newblock Filter bank frame expansions with erasures.
\newblock {\em IEEE Transactions on Information Theory}, 48(6):1439--1450,
  2002.

\bibitem{leng2011optimal}
Jinsong Leng, Deguang Han, and Tingzhu Huang.
\newblock Optimal dual frames for communication coding with probabilistic
  erasures.
\newblock {\em IEEE transactions on signal processing}, 59(11):5380--5389,
  2011.

\bibitem{li2017somee}
Dongwei Li and Jinsong Leng.
\newblock On some new inequalities for fusion frames in \protect{H}ilbert
  spaces.
\newblock {\em Mathematical Inequalities \& Applications}, 20(3):889--900,
  2017.

\bibitem{li2018some}
Dongwei Li and Jinsong Leng.
\newblock On some new inequalities for continuous fusion frames in
  \protect{H}ilbert spaces.
\newblock {\em Mediterranean Journal of Mathematics}, 15(4):173, 2018.

\bibitem{li2018frame}
Dongwei Li, Jinsong Leng, Tingzhu Huang, and Qing Gao.
\newblock Frame expansions with probabilistic erasures.
\newblock {\em Digital Signal Processing}, 72:75--82, 2018.

\bibitem{li2012some}
Jian-Zhen Li and Yu-Can Zhu.
\newblock Some equalities and inequalities for g-bessel sequences in
  \protect{H}ilbert spaces.
\newblock {\em Applied Mathematics Letters}, 25(11):1601--1607, 2012.

\bibitem{poria2017some}
Anirudha Poria.
\newblock Some identities and inequalities for
  \protect{H}ilbert-\protect{S}chmidt frames.
\newblock {\em Mediterranean Journal of Mathematics}, 14(2):59, 2017.

\bibitem{sun2017nonlinear}
Qiyu Sun and Wai-Shing Tang.
\newblock Nonlinear frames and sparse reconstructions in \protect{B}anach
  spaces.
\newblock {\em Journal of Fourier Analysis and Applications}, 23(5):1118--1152,
  2017.

\bibitem{vashisht2016weaving}
LK~Vashisht et~al.
\newblock Weaving properties of generalized continuous frames generated by an
  iterated function system.
\newblock {\em Journal of Geometry and Physics}, 110:282--295, 2016.

\bibitem{xiao2015erasures}
Xiang-chun Xiao, Yu-can Zhu, and Ming-ling Ding.
\newblock Erasures and equalities for fusion frames in hilbert spaces.
\newblock {\em Bulletin of the Malaysian Mathematical Sciences Society},
  38(3):1035--1045, 2015.

\bibitem{zhao2006perturbation}
Ping Zhao, Chun Zhao, and Peter~G Casazza.
\newblock Perturbation of regular sampling in shift-invariant spaces for
  frames.
\newblock {\em IEEE Transactions on Information Theory}, 52(10):4643--4648,
  2006.

\bibitem{zhu2010note}
Xiuge Zhu and Guochang Wu.
\newblock A note on some equalities for frames in \protect{H}ilbert spaces.
\newblock {\em Applied Mathematics Letters}, 23(7):788--790, 2010.

\end{thebibliography}

\end{document}